\documentclass[12pt]{amsart}
\usepackage{amssymb, amstext, amscd, amsmath}
\usepackage{mathtools, xypic, rotating, mathabx}
\usepackage{color}

\makeatletter
\def\@cite#1#2{{\m@th\upshape\bfseries%
[{#1\if@tempswa{\m@th\upshape\mdseries, #2}\fi}]}}
\makeatother
\theoremstyle{plain}
\newtheorem{thm}{Theorem}[section]
\newtheorem{cor}[thm]{Corollary}

\newtheorem{lem}[thm]{Lemma}

\theoremstyle{definition}

\newtheorem*{acknow}{Acknowledgements}
\numberwithin{equation}{section}
\mathtoolsset{centercolon}
\newcommand{\bC}{{\mathbb{C}}}

\newcommand{\B}{{\mathcal{B}}}

\newcommand{\ep}{\varepsilon}
\renewcommand{\phi}{\varphi}
\def\si{\sigma}

\newcommand{\FOR}{\text{ for }}

\newcommand{\foral}{\text{ for all }}
\newcommand{\qand}{\quad\text{and}\quad}

\newcommand{\ol}{\overline}

\newcommand{\sumoplus}{\operatornamewithlimits{\sum \strut^\oplus}}
\newcommand{\fp}{\check{\bigast}}

\begin{document}

\title{A proof of Boca's Theorem}

\author[K.R. Davidson]{Kenneth R. Davidson}
\address{Pure Mathematics Department\\
University of Waterloo\\
Waterloo, ON\; N2L--3G1\\
CANADA}
\email{krdavids@uwaterloo.ca}

\author[E.T.A. Kakariadis]{Evgenios T.A. Kakariadis}
\address{School of Mathematics and Statistics\\
Newcastle University\\
Newcastle upon Tyne, NE1 7RU\\
UK}
\email{evgenios.kakariadis@ncl.ac.uk}

\begin{abstract}
We give a general method of extending unital completely positive maps to amalgamated free products of C*-algebras.
As an application we give a dilation theoretic proof of Boca's Theorem.
\end{abstract}



\maketitle

\section{Introduction}

The amalgamated free product of C*-algebras has become a standard construction in the subject.
It is central in studying group C*-algebras, free probability, quantum information, dynamical systems on C*-algebras, and the Connes Embedding Conjecture  via its connection to Kirchberg's conjecture and also to Tsirelson's conjectures. 
Boca's result on completely positive maps on amalgamated free products has become a useful technical tool in this endeavour.
A comprehensive survey would be too time consuming for this short note.

Avitzour \cite{Avi82} showed that states on two C*-algebras can be extended to a state on the free product of the two C*-algebras.
Boca's Theorem \cite{Boc91} shows that unital completely positive maps on two C*-algebras agreeing on a common C*-subalgebra can be extended to a unital completely positive map on their amalgamated free product.
However his method requires some additional structure, and yields additional structure.
Here we give a simple proof of the basic result that does not require these added hypotheses,
and then we explain how the argument can be modified to yield Boca's actual theorem.

To fix notation, let $\{A_i\}_{i \in I}$ be a family of unital C*-algebras with a common unital C*-subalgebra $B$, i.e., there are unital imbeddings $\ep_i: B \to A_i$. 
The amalgamated free product $\fp_B A_i$ is an appropriate completion of the $*$-algebraic amalgamated free product $\bigast_B A_i$.
It is the universal C*-algebra generated by $\bigcup_i \phi_i(A_i)$, for imbeddings $\{\phi_i\}_{i \in I}$ with $\phi_i \ep_i = \phi_j \ep_j$, such that: every family of $*$-rep\-re\-sent\-ations 
\[ \{ \psi_i : A_i \to \B(H) \mid \psi_i \ep_i = \psi_j \ep_j \FOR i,j \in I \} \]
lifts to a $*$-rep\-re\-sent\-ation $\pi : \fp_B A_i \to \B(H)$ such that $\psi_i=\pi\phi_i$.

If we assume the existence of conditional expectations $E_i : A_i \to B$ for $i\in I$, 
then the $*$-algebraic amalgamated free product, as a linear space, takes the form
\begin{equation}\label{eq:form}
\bigast \hspace{-3pt} \strut_B A_i = B \oplus \sumoplus_{i_1 \cdots i_n \in S} ( \ker E_{i_1} \otimes_B \cdots \otimes_B \ker E_{i_n})
\end{equation}
where $\otimes_B$ denotes the bimodule tensor product, and
\begin{equation}\label{eq:set}
S := \{ i_1 \cdots i_n : n\ge1,\  i_1, \dots, i_n \in I, i_1 \neq \cdots \neq i_n \}
\end{equation}
is the set of non-empty words in the alphabet $I$ that contain no subwords of the form $ii$.
See \cite{Avi82} or \cite[Proof of Theorem 3.1]{Coh64}.
It is natural to use (\ref{eq:form}) for extending linear maps \emph{canonically}.
Indeed let $\{ \Phi_i : A_i \to \B(H) \}_{i \in I}$ be unital completely positive maps which restrict to a common $*$-rep\-re\-sent\-ation $\rho$ of $B$.
Then one can directly define a linear map $\Phi$ on $\bigast_B A_i$ given by
\begin{equation}\label{eq:ce}
\Phi(b) =  \rho(b)  \qand  \Phi( a_{i_1} \ast \cdots \ast a_{i_n}) := \Phi_{i_1}(a_{i_1}) \cdots \Phi_{i_n}(a_{i_n})
\end{equation}
for $b \in B$ and $a_{i_k} \in \ker E_{i_k}$ when $i_1\cdots i_n \in S$.
Boca shows that $\Phi$ extends to a unital completely positive map of $\fp_B A_i$.
Applying this to $B= \bC$ yields the fact that unital completely positive maps of $A_i$ into a common Hilbert space extend to a unital completely positive map of $\fp A_i$.

To achieve his results, Boca verifies matrix inequalities for elements in $\bigast_B A_i$.
However this line of reasoning does not apply in the absence of expectations.
Here we present a method that tackles this problem by exploiting the ideas of \cite{DFK17}.
This produces an explicit Stinespring dilation of the desired completely positive map.
Boca \cite{Boc93} produces an explicit Stinespring dilation by different methods when expectations are available.
With additional care, our arguments also provide an alternative proof of Boca's result.

Our strategy is to dilate the maps $\{ \Phi_i \}_{i \in I}$ to $*$-rep\-re\-sent\-ations $\{ \pi_i: A_i \to\B(K) \}_{i \in I}$ on a common Hilbert space $K$ that agree on $B$.
The universal property will then provide a $*$-rep\-re\-sent\-ation $\pi$ of $\fp_B A_i$.
The compression of $\pi$ to $H$ yields the desired completely positive map $\Phi$ (Theorem \ref{T:boca}).
We further use this to construct a unital completely positive extension in the case where the $\Phi_i$ agree just as linear maps on $B$. 
Under the additional structure of \cite{Boc91}, we can construct $\Phi$ so that (\ref{eq:ce}) holds, and thus it coincides with Boca's map.
Actually we do more here.
Given a sub-family of expectations $\{E_j : A_j \to B\}_{j \in J}$ for $J \subset I$, we can construct $\Phi$ that satisfies (\ref{eq:ce}) for $i_1, \dots, i_n \in J$ (Theorem \ref{T:boca2}).

\section{Preliminaries}

Let $I$ be an index set of arbitrary cardinality. 
Fix a family $\{A_i\}_{i \in I}$ of unital C*-algebras that contain a common unital C*-subalgebra $B$ in the sense that there are faithful unital imbeddings $\ep_i : B \to A_i$ for $i\in I$.
We denote by $\bigast_B A_i$ the \emph{$*$-algebraic amalgamated free product} with canonical imbeddings $\phi_i : A_i \to \bigast_B A_i$.
The \emph{amalgamated free product} $\fp_B A_i$ of C*-algebras $\{A_i\}_{i\in I}$ is then its quotient completion with respect to the seminorm
\[
\|x\| := \sup \{ \|\pi(x)\| : \text{ $\pi$ is a $*$-rep\-re\-sent\-ation of $\bigast_B A_i$} \}.
\]
The supremum is finite since the $\pi \phi_i$ are $*$-rep\-re\-sent\-ations of $A_i$.
The existence of $\fp_B A_i$ is then routine; e.g.\ \cite[p. 88]{DFK17}.
Blackadar \cite[Theorem 3.1]{Bla78} shows that such a $\pi$ can be constructed so that the $\pi \phi_i$ are isometric.
Therefore the canonical imbeddings $A_i \to \fp_B A_i$ for $i\in I$ are isometric; and henceforth we will suppress their use.
Notice that a $*$-rep\-re\-sent\-ation $\pi$ of $\fp_B A_i$ satisfies $\pi \ep_i = \pi\ep_j$ for all $i, j \in I$.

Pedersen \cite{Ped99} shows that if $B \subset C_1 \subset A_1$ and $B \subset C_2 \subset A_2$ are unital inclusions of C*-algebras, then the natural map of $C_1 \check\bigast C_2$ into $A_1 \check\bigast A_2$ given by the universal property is injective.
An alternative proof was given later by Armstrong, Dykema, Exel and Li \cite{ADEL04}.
This result can be extended to free products of finitely many C*-algebras.

With Fuller, the authors gave a direct extension of these arguments in \cite[Lemma 5.3.18]{DFK17}.
Even though \cite[Lemma 5.3.18]{DFK17} treats the finite case, the reasoning can be applied verbatim to tackle arbitrary families.
We sketch the proof because we wish to establish some notation.

The following elementary fact will be used repeatedly. 
Suppose that $\Phi$ is a unital completely positive map of a C*-algebra $A$ into $\B(H)$, and $B$ is a C*-subalgebra of $A$ such that $\Phi|_B$ is a $*$-representation.
Then given any Stinespring dilation $\pi: A \to \B(H \oplus H')$ of $\Phi$, the space $H$ reduces $\pi(B)$.
This follows because the compression of a $*$-representation to a subspace is multiplicative only when the subspace is invariant.

\begin{lem}\label{L:DFK}
Let $B$ be a unital C*-algebra.
Let $\{A_i\}_{i \in I}$ be a family of unital C*-algebras with faithful unital imbeddings $\ep_i: B \to A_i$, 
and suppose that there are $*$-rep\-re\-sent\-ations $\rho_i : B \to \B(H_i)$ for $i\in I$.
Then there are a Hilbert space $K$ containing $\sum_{i \in I}^\oplus H_i$ and $*$-rep\-re\-sent\-ations $\pi_i : A_i \to \B(K \ominus H_i)$ such that the $\pi_i$ agree on $B$ in the sense that 
\[
\rho_i \oplus \pi_i \ep_i = \rho_j \oplus \pi_j \ep_j \foral i, j \in I.
\]
\end{lem}

\begin{proof}
Recall that $S$ is the set of non-empty words in $I$ that contain no subwords of the form $ii$. 
For a word $i_1 \cdots i_n$ in $I$, we set 
\begin{equation*}
s( i_1 \cdots i_n) := i_n \qand | i_1 \cdots i_n| := n.
\end{equation*}
We will recursively define Hilbert spaces $H_w$ for $w\in S$ with $|w|\ge 2$,  
$*$-representations $\rho_w:B \to \B(H_w)$, and for $u\in S$ and $i \ne s(u)$, 
we will construct a $*$-representation $\pi_{i,u}:A_i \to \B(H_u \oplus H_{ui})$ so that $\pi_{i,u}\ep_i = \rho_u \oplus \rho_{ui}$. 

By hypothesis, the $\rho_j$ are defined for $j\in I$. Consider them as unital completely positive maps of $B$ into $\B(H_j)$.
For each $i\ne j$, use Arveson's extension theorem to extend $\rho_j$ to a completely positive map of $A_i$ into $\B(H_j)$, 
and then apply Stinespring's dilation theorem\footnote{
Alternatively we may use the more elementary result \cite[Proposition 2.10.2]{Dix77}. For reasons that will become clear later, we do not follow this route.}
to obtain a $*$-rep\-re\-sent\-ation $\pi_{i,j}:A_i\to\B(H_j \oplus H_{ji})$.
By the remark preceding the lemma, $\pi_{i,j}\ep_i(B)$ reduces $H_j$ (and hence also $H_{ji}$). 
So we may define a $*$-rep\-re\-sent\-ation $\rho_{ji}:= P_{H_{ji}}\pi_{i,j}\ep_i$. We then have $\pi_{i,j}\ep_i = \rho_j \oplus \rho_{ji}$.

Now consider $k\ge2$.
Suppose that the $H_w$ are defined for $w \in S$ with $|w| \le k$, 
that $\rho_w:B \to \B(H_w)$ are $*$-representations, and that there are $*$-rep\-re\-sent\-ations 
$\pi_{i,u}$ of $A_i$ on $H_u \oplus H_{ui}$ whenever $u \in S$ with $|u|<k$ and $s(u) \ne i$ such that 
$\pi_{i,u}\ep_i = \rho_u \oplus \rho_{ui}$. 
If $|w|=k$ and $i\ne s(w)$, use Arveson's extension theorem to extend $\rho_w$ to a completely positive map of $A_i$ into $\B(H_w)$.
Then apply Stinespring's dilation theorem to obtain a $*$-rep\-re\-sent\-ation $\pi_{i,w}:A_i\to\B(H_w \oplus H_{wi})$.
By the remark preceding the lemma, $\pi_{i,w}\ep_i(B)$ reduces $H_w$ (and hence also $H_{wi}$). 
So we may define a $*$-rep\-re\-sent\-ation $\rho_{wi}:= P_{H_{wi}}\pi_{i,w}\ep_i$. We then have $\pi_{i,w}\ep_i = \rho_w \oplus \rho_{wi}$.
This completes the induction.

Set $K = \sum_{w\in S}^\oplus H_w$. 
Define $\pi_i = \sum^\oplus_{s(w)\ne i} \pi_{i,w}$, which is a $*$- rep\-re\-sent\-ation of $A_i$ on
\[ \sumoplus_{s(w)\ne i} H_w \oplus H_{wi} = \sumoplus_{u \in S,\ u \ne i} H_u = K \ominus H_i .\]
Moreover 
\[ \pi_i\ep_i = \sumoplus_{s(w)\ne i}  \rho_w \oplus \rho_{wi} = \sumoplus_{u \in S,\ u \ne i}  \rho_u .\]
Thus $\rho_i \oplus \pi_i\ep_i = \sum^\oplus_{u \in S} \rho_u$ is independent of $i$.
\end{proof}

As a consequence we get the imbedding of \cite{Ped99, ADEL04, DFK17} for arbitrary families of C*-algebras.

\begin{cor}\label{C:incl}
Let $\{C_i\}_{i \in I}$ and $\{A_i\}_{i \in I}$ be families of unital C*-algebras such that $B \subset C_i \subset A_i$ for a common unital C*-subalgebra $B$.
Then $\fp_B C_i \subset \fp_B A_i$ via the natural inclusion map.

In particular, if $J$ is a non-empty subset of $I$, then $\fp_B^{i \in J} A_i \subset \fp_B A_i$ via the natural inclusion map.
\end{cor}

\begin{proof}
For the first part, let $\si : \fp_B C_i \to \B(H)$ be a faithful $*$-rep\-re\-sent\-ation.
We can find Hilbert spaces $H_i$ and extend every $\si_i := \si|_{C_i}$ to a $*$-rep\-re\-sent\-ation $\widetilde{\si}_i \colon A_i \to \B(H \oplus H_i)$.
Since $\si_i \ep_i$ is a $*$-rep\-re\-sent\-ation of $B$ on $H$, we can decompose $\widetilde{\si}_i \ep_i = \si_i \ep_i \oplus \rho_i$.
Hence we can apply Lemma \ref{L:DFK} to $\rho_i : B \to \B(H_i)$ and obtain the $*$-rep\-re\-sent\-ations $\pi_i \colon A_i \to \B(K \ominus H_i)$ such that
\[
\rho_i \oplus \pi_i \ep_i = \rho_j \oplus \pi_j \ep_j. 
\]
Define $\tau_i : A_i \to \B(H \oplus K)$ by $\tau_i = \widetilde{\si}_i \oplus \pi_i$.
Since $\si_i \ep_i = \si_{j} \ep_j$, this construction yields
\begin{align*}
\tau_i \ep_i
& =
\si_i \ep_i \oplus \rho_i \oplus \pi_i \ep_i
=
\si_j \ep_j \oplus \rho_j \oplus \pi_j \ep_j
=
\tau_{j} \ep_j.
\end{align*}
The universal property of the amalgamated free product then yields a $*$-rep\-re\-sent\-ation of $\fp_B A_i$ that extends $\si$.
On the other hand every $*$-rep\-re\-sent\-ation of $\bigast_B A_i$ restricts to a $*$-rep\-re\-sent\-ation of $\bigast_B C_i$.
Hence the canonical inclusion map $\bigast_B C_i \hookrightarrow \bigast_B A_i$ extends to an isometry on their completions.

We can now apply this to the family $\{C_i\}_{i \in I}$ given by
\[
C_i
=
\begin{cases}
A_i & \text{ if } i \in J, \\
B & \text{ if } i \notin J .
\end{cases}
\]
Notice that
\[
\bigast \hspace{-3pt} \strut_B C_i = 
(\bigast \hspace{-3pt} \strut_B^{i \in J} A_i) \bigast \hspace{-3pt} \strut_B (\bigast \hspace{-3pt} \strut_B^{i \notin J} B) = 
(\bigast \hspace{-3pt} \strut_B^{i \in J} A_i) \bigast\hspace{-3pt} \strut_B B = 
(\bigast \hspace{-3pt} \strut_B^{i \in J} A_i). 
\]
The second claim then follows, since the $*$-rep\-re\-sent\-ations of $\bigast_B^{i \in J} A_i$ and $\bigast_B C_i$ coincide.
\end{proof}

The free product construction can be readily formulated when $A_i$ are possibly non-selfadjoint.
They have been studied first by Duncan \cite{Dun08}.
Their theory was later established in \cite[Section 5.3]{DFK17}, and exhibited in an alternative way by Dor-On and Salomon \cite[Section 4]{DS17}. 
Dor-On and Salomon \cite[Proposition 4.3]{DS17} develop similar dilation techniques and establish Corollary \ref{C:incl} in that generality.

\section{Free products of ucp maps}

We first establish that there is always a common extension of unital completely positive maps on each $A_i$ to the amalgamated free product
when they restrict to a common $*$-rep\-re\-sent\-ation on $B$. 
As this does not require expectations, it is a natural extension of Boca's result.

\begin{thm}\label{T:boca}
Let $B$ be a unital C*-algebra.
Let $\{A_i\}_{i \in I}$ be a family of unital C*-algebras and let $\ep_i: B \to A_i$ be faithful unital imbeddings. 
Let $\Phi_i: A_i \to \B(H)$ be unital completely positive maps which restrict to a common $*$-rep\-re\-sent\-ation of $B$. 
Then there is a unital completely positive map $\Phi: \fp_B A_i \to \B(H)$ such that $\Phi|_{A_i} = \Phi_i$ for all $i \in I$.
\end{thm}

\begin{proof}
By hypothesis $\rho_0 = \Phi_i\ep_i$ is a $*$-rep\-re\-sent\-ation of $B$ which is independent of $i\in I$.
Use Stinespring's Theorem to dilate each $\Phi_i$ to a $*$-rep\-re\-sent\-ation $\si_i$ of $A_i$ on $H \oplus H_i$ such that $\Phi_i(a) = P_H \si_i(a)|_H$ for $a \in A_i$.
Since $\Phi_i\ep_i = \rho_0$ is a $*$-rep\-re\-sent\-ation, it follows that $\si_i \ep_i = \rho_0 \oplus \rho_i$ is a direct sum of $*$-rep\-re\-sent\-ations of $B$ on $H$ and $H_i$.
Now we apply Lemma \ref{L:DFK} to the family of $*$-representations $\rho_i : B \to \B(H_i)$ for $i\in I$, and obtain $*$-rep\-re\-sent\-ations $\pi_i: A_i\to \B(K \ominus H_i)$ such that
\[
\rho_i \oplus \pi_i \ep_i = \rho_j \oplus \pi_j \ep_j.
\]
Then the  $*$-rep\-re\-sent\-ations
\[
\tau_i := \si_i\oplus \pi_i : A_i \to \B(H \oplus K)
\]
agree on $B$, i.e. $\tau_i \ep_i = \tau_j \ep_j$.
By the universal property of free products, there is a $*$-rep\-re\-sent\-ation
\[
\tau : \fp_B A_i \to \B(H \oplus K) \text{ with } \tau|_{A_i} = \tau_i .
\]
Then $\Phi = P_H \tau |_H$ is the required completely positive map.
\end{proof}

We can extend Theorem \ref{T:boca} to the case where the $\Phi_i$ agree just as linear maps on $B$.

\begin{thm}\label{T:boca cp}
Let $B$ be a unital C*-algebra.
Let $\{A_i\}_{i \in I}$ be a family of unital C*-algebras and let $\ep_i: B \to A_i$ be faithful unital imbeddings. 
Let $\Phi_i: A_i \to \B(H)$ be unital completely positive maps which restrict to a common linear map of $B$. 
Then there is a unital completely positive map $\Phi: \fp_B A_i \to \B(H)$ such that $\Phi|_{A_i} = \Phi_i$ for all $i \in I$.
\end{thm}

\begin{proof}
By hypothesis  $\Phi_0 = \Phi_i\ep_i$ is a unital completely positive map of $B$ into $\B(H)$ which is independent of $i\in I$.
Let $\si_i : A_i \to \B(H \oplus H_i)$ be a Stinespring dilation of each $\Phi_i$ and set $M_i = \ol{\si_i \ep_i(B) H}$. 
This is the minimal reducing subspace for $\si_i\ep_i(B)$, and thus determines a minimal Stinespring dilation $\rho_i$ of $\Phi_0$, namely
\[
\rho_i := P_{M_i} \si_i\ep_i : B \to \B(M_i).
\]
Fix some $i_0\in I$.
By uniqueness of the minimal dilation, there are unitary operators $U_i: M_{i_0} \to M_i$ such that
\[
U_i|_H = I_H \qand U_i^* \rho_i(b) U_i = \rho_{i_0}(b) \foral b\in B.
\]
Define unital completely positive maps of $A_i$ into $\B(M_{i_0})$ by
\[
\Psi_i(a) = U_i^* P_{M_i} \si_i(a) U_i \foral a \in A_i.
\]
Notice that they restrict to the common $*$-rep\-re\-sent\-ation $\rho_{i_0}$ on $B$. 
Thus by Theorem \ref{T:boca}, there is a unital completely positive map $\Psi : \fp_B A_i \to \B(M_{i_0})$ such that $\Psi|_{A_i} = \Psi_i$.
For $a \in A_i$ we have
\[
P_H U_i^* P_{M_i} \si_i(a) U_i|_H
=
P_H \si_i(a)|_H = \Phi_i(a) .
\]
Therefore the compression $\Phi = P_H \Psi|_H$ is the required map.
\end{proof}

We can modify the construction of Theorem \ref{T:boca} to prove Boca's result.
We use the following lemma.

\begin{lem} \label{L:ker}
Begin with the same setup and notation as in Lemma~$\ref{L:DFK}$ and its proof.
Furthermore assume that there is a subset $J \subset I$ for which there are conditional expectations $E_j$ of $A_j$ onto $B$ for $j\in J$.
Then there are $*$-rep\-re\-sent\-ations $\pi_i : A_i \to \B(K \ominus H_i)$ such that 
\[
\rho_i \oplus \pi_i \ep_i = \rho_j \oplus \pi_j \ep_j \foral i, j \in I
\]
with the additional property that $\pi_j(a) H_w \subset H_{wj}$ whenever $a \in \ker E_j$ for $j\in J$ and $s(w)\ne j$.
\end{lem}

\begin{proof}
The only change to the proof of Lemma~\ref{L:DFK} is in the construction of the $*$-rep\-re\-sent\-ations $\pi_{j,w}$ for $j\in J$ and $s(w)\ne j$. 
For these $(j, w)$ we will use the completely positive map $\rho_w E_j$ from $A_j$ into $\B(H_w)$.
Then use Stinespring's Dilation Theorem to obtain the $*$-rep\-re\-sent\-ation $\pi_{j,w}$ of $A_j$ into $\B(H_w\oplus H_{wj})$ such that $\rho_w E_j = P_{H_w} \pi_{j,w} |_{H_w}$.
Since $\ker E_j$ is in the kernel of this completely positive map, we obtain the matrix form with respect to the decomposition $H_w\oplus H_{wj}$
\[ 
\pi_j(a)|_{H_w \oplus H_{wj}} = \pi_{j, w}(a) = \begin{bmatrix} 0 & *\\ \ast&*\end{bmatrix} 
\]
for $a \in \ker E_j$. This means that $\pi_j(a) H_w \subset H_{wj}$.
\end{proof}

We are now able to give the dilation theoretic proof that generalizes \cite[Theorem 3.1]{Boc91}.
This provides an explicit Stinespring dilation, which Boca accomplished in his setting in \cite{Boc93}.

\begin{thm}\label{T:boca2}
Let $\{A_i\}_{i \in I}$ be unital C*-algebras containing a common unital C*-subalgebra $B$,
and suppose that there are conditional expectations $E_j$ of $A_j$ onto $B$ for $j \in J \subset I$.
Let $\Phi_i: A_i\to\B(H)$ be unital completely positive maps which restrict to a common $*$-rep\-re\-sent\-ation of $B$. 
Then there is a unital completely positive map $\Phi: \fp_B A_i \to\B(H)$ such that $\Phi|_{A_i}=\Phi_i$ for all $i\in I$, and 
\begin{equation}\label{eq:prod}
\Phi(a_n\cdots a_1) = \Phi_{j_n} (a_n) \cdots \Phi_{j_1} (a_1) 
\end{equation}
when all $j_k\in J$, $a_k\in \ker E_{j_k} \subset A_{j_k}$, and $j_1 \neq \cdots \neq j_n$.
\end{thm}

\begin{proof}
The construction is identical to the proof of Theorem~\ref{T:boca} except that we use the refinement in Lemma~\ref{L:ker}. 
Therefore we now have $*$-rep\-re\-sent\-ations
\[
\tau_i := \si_i \oplus \pi_i : A_i \to \B \big(H \oplus \sumoplus_{w \in S} H_w \big)
\]
such that for $h\in H$ and $a\in A_i$,
\[
\tau_i(a)h \in \Phi_i(a)h + H_i \foral i \in I,
\]
and in addition
\[
\pi_j(\ker E_j) H_w \subset H_{wj} \text{ when } j \in J \text{ and } s(w) \neq j.
\]
It remains only to verify the identity (\ref{eq:prod}). 
By construction we have 
\[ 
 \Phi(a_n\cdots a_1) = P_{H} \tau_{j_n}(a_n) \cdots \tau_{j_1}(a_1) |_H .
\]
We will show by induction that for $h\in H$,
\[
 \tau_{j_n}(a_n) \cdots \tau_{j_1}(a_1) h
 \in \Phi_{j_n} (a_n) \cdots \Phi_{j_1} (a_1) h + \sumoplus_{s(w)= j_n} H_w
\]
when all $j_k\in J$, $a_k\in \ker E_{j_k} \subset A_{j_k}$, and $j_1 \neq \cdots \neq j_n$.
This holds for $n=1$. 
Assuming the result for $n-1$, we see that
\begin{align*}
 \tau_{j_{n}}(a_{n}) \cdots \tau_{j_1}(a_1) h & \in 
 \tau_{j_{n}}(a_{n}) \Big(\Phi_{j_{n-1}} (a_{n-1}) \cdots \Phi_{j_1} (a_1) h + \sumoplus_{s(w)= j_{n-1}} H_w \Big)\\
 &\subseteq \Phi_{j_{n}} (a_{n}) \cdots \Phi_{j_1} (a_1) h + H_{j_{n}} + \sumoplus_{s(w)= j_{n-1}} H_{wj_{n}} \\
 &\subseteq \Phi_{j_{n}} (a_{n}) \cdots \Phi_{j_1} (a_1) h + \sumoplus_{s(w)= j_{n}} H_{w} .
\end{align*}
By induction, this holds for all $n$.
Compression to $H$ yields that
\[ 
\Phi(a_n\cdots a_1) = \Phi_{j_n} (a_n) \cdots \Phi_{j_1} (a_1) .\qedhere
\]
\end{proof}

\bigskip

\begin{acknow}
The authors acknowledge support from the London Ma\-thematical Society.
Kenneth Davidson visited Newcastle University under a Scheme 2 LMS grant (Ref: 21602).
The authors would like to thank Steven Power (Lancaster University) and Ivan Todorov (Queen's University of Belfast) for supporting this UK visit.

We thank Chris Ramsey for asking a question which led to the inclusion of Theorem~\ref{T:boca cp}.

We also thank the referee for a meticulous reading of our paper.
\end{acknow}


\end{document}